\documentclass[11.5pt]{amsart}
\usepackage{amssymb} 
\newtheorem{thm}{Theorem}[section]
\newtheorem{cor}[thm]{Corollary}
\newtheorem{lem}[thm]{Lemma}

\newtheorem{defn}[thm]{Definition}
\newtheorem{rem}[thm]{Remark}

\theoremstyle{question}

\numberwithin{equation}{section}

\begin{document}
\title[minimal number of generators]{}%
\begin{center}
{\bf\Large Minimal Number of Generators and Minimum Order of a
Non-Abelian Group whose Elements Commute with Their Endomorphic
Images}
\end{center}
\vspace{1cm}
\begin{center}
{\bf A. Abdollahi$~^*$} \;\;\;  {\bf A. Faghihi} \;\;\; and \;\;\;
{\bf A. Mohammadi Hassanabadi}
\\ Department of Mathematics,\\ University of Isfahan,\\ Isfahan
81746-73441,\\ Iran.
\end{center}
\thanks{$~^*$ Corresponding Author. e-mail: {\tt a.abdollahi@math.ui.ac.ir}}
\thanks{This work was supported partially by the Center of Excellence for Mathematics, University of Isfahan.}

\begin{abstract} A group in which every element commutes
with its endomorphic images is called an $E$-group.  If $p$ is a
prime number, a $p$-group $G$
 which is an $E$-group is called a $pE$-group.
 Every abelian group is obviously an $E$-group.
  We prove that every 2-generator $E$-group is abelian
 and that all 3-generator $E$-groups are nilpotent of class at
 most 2.  It is also proved that every infinite 3-generator $E$-group is abelian. We conjecture that every finite
 3-generator $E$-group should  be abelian.  Moreover we   show that  the
 minimum order of a non-abelian $pE$-group is $p^8$ for any odd
 prime number $p$ and this order is $2^7$ for $p=2$. Some of these
 results are proved for a class wider than the class of
 $E$-groups.
\end{abstract}
\maketitle
\section{\bf Introduction and results }
A group in which each element commutes with its endomorphic
images is called an ``$E$-group". It is well-known (see e.g.,
\cite{m2}) that a group $G$ is an $E$-group if and only if the
near-ring generated by the endomorphisms of $G$ in the near-ring
of maps on  $G$ is a ring. \\
Since in an $E$-group every element commutes with its image under
inner automorphisms, every $E$-group is a 2-Engel
 group, and so they are nilpotent of class  at most 3 (see \cite{l}, or \cite[Theorem 12.3.6]{r1}).
  Throughout the paper $p$ denotes a prime number. We
 call an $E$-group which is also a $p$-group, a $pE$-group.
 Since a finite $E$-group can be written as a direct product of
 its Sylow subgroups, and any direct factor of an $E$-group is an
 $E$-group \cite{m3}, so we need only consider  $pE$-groups.\\
  The first examples of non-abelian $pE$-groups
 are due to
  R. Faudree \cite{f}, which are defined as follows:
\begin{align*} G=\langle a_1,a_2,a_3,a_4 \;|\; a_i^{p^2}=1, [a_i,a_j,a_k]=1,
\;\; i,j,k\in\{1,2,3,4\} &\\ [a_1,a_2]=a_1^p,  [a_1,a_3]=a_3^p,
[a_1,a_4]=a_4^p, [a_2,a_3]=a_2^p, [a_2,a_4]=1, [a_3,a_4]=a_3^p~
\rangle,
\end{align*} where $p$ is any odd prime number.
 Note that the above example is false for
 $p=2$. This is because, when $p=2$ then the map $\alpha$ defined by
$$a_1^\alpha=a_1^{-1}a_2a_4, \;  a_2^\alpha=a_3, \;
a_3^\alpha=a_4, \; a_4^\alpha=a_1a_4$$ can be extended to an
endomorphism of $G$. But $[a_3^\alpha,a_3]=[a_4,a_3]\neq 1$, so
that $G$ is not an  $E$-group.\\
All known  examples  of non-abelian $E$-groups have nilpotency
class 2 (see \cite{c}, \cite{cf}, \cite{f}, \cite{m1}). In this
paper, we see new examples of $E$-groups.\\
 A. Caranti posed the question \cite[Problem 11.46
 a]{k} of whether there exists a finite $3E$-group  of nilpotency class
 3. (Note that every $E$-group without elements of order 3 is
 nilpotent of class at most 2).\\
Some partial (negative) answers to this question are as follows:
finite $3E$-groups of exponent dividing 9 are nilpotent of class
at most 2 \cite{m3}; every   2-generator $E$-group is nilpotent of
class at most 2 (since they are 2-Engel). Here we concentrate on
the following questions:\\
(1) \; What is the least number of generators of a finitely
generated non-abelian
 $E$-group?\\
 (2) \; What is the minimum order of a finite non-abelian
 $pE$-group?\\
We prove
\begin{thm} \label{thm3gen} Every finite $3$-generator $E$-group is nilpotent of class at most $2$.
\end{thm}
\begin{thm} \label{thm3^{10}} Let $G$ be a $3E$-group. If $|G|\leq 3^{10}$, then
$G$ is nilpotent of class at most $2$.
\end{thm}
As we said, it is easily seen that every 2-generator $E$-group is
nilpotent of class at most 2. In fact an stronger result holds,
namely:
\begin{thm}\label{2-gen} {\rm (i)} \; Every $2$-generator group is abelian if
and only if it is an $E$-group.\\
{\rm (ii)} \; Every infinite $3$-generator group is abelian if and
only it is an $E$-group.
\end{thm}
Thus,   by Theorem \ref{2-gen} and Faudree's examples of
$E$-groups, the minimal number of generators of a non-abelian
$E$-group is  $3$ or $4$. We conjecture that this minimal number
must be 4, that is every 3-generator $E$-group is abelian.\\
In  response to question (2), we prove
\begin{thm}\label{thmp^6} For any prime number $p$, every  $pE$-group of order at most  $p^6$ is
abelian.
\end{thm}
\begin{thm}\label{thmp^7} {\rm (i)} \; For any odd prime number $p$, every  $pE$-group of order at most  $p^7$ is
abelian.\\
{\rm (ii)} \; There exist non-abelian $2E$-groups of order $2^7$.
\end{thm}
As a result of  Theorems \ref{thmp^6} and \ref{thmp^7} and
Faudree's examples of $E$-groups,  we conclude that the minimum
order of a finite non-abelian $pE$-group is $p^8$, for any odd
prime number $p$ and this order is $2^7$ for $p=2$.
\section{\bf Preliminary definitions and results}
 Let $G$ be a group, $H$ be an element or a subgroup of $G$ and $n$ be a positive integer. We denote by $G'$, $\Phi(G),
 G^n$, $\gamma_3(G)$,
 $Z(G)$, $Z_2(G)$, $\text{exp}(G)$, $Aut(G)$,
 $End(G)$, $C_G(H)$, $Q_8$, and $C_n$, respectively the derived subgroup, the Frattini subgroup, the
 subgroup generated by $n$-powers of the elements,
 the third term of the lower
 central series, the center, the second center, the exponent, the automorphism group, the set of
 endomorphisms of $G$, the centralizer of $H$ in $G$, the quaternion group of order 8, and the cyclic group of order $n$.
For a finite $p$-group $G$ and a positive integer $n$, we denote
by $\Omega_n(G)$ the subgroup generated by elements $x$ such that
$x^{p^n}=1$. If $G$ is a
 nilpotent group, we denote by $\text{cl}(G)$ the nilpotency class of $G$. If $G$ is a finite
 group, $d(G)$ will denote the minimum number of generators of
 $G$. If $a, b$ and $c$ are elements of a group, we
denote by $[a,b]$ the commutator $a^{-1}b^{-1}ab$ and we define
$[a,b,c]=[[a,b],c]$, as usual. An automorphism $\alpha$ of a group $G$ is called central if $x^{-1}x^{\alpha} \in Z(G)$ for all $x\in G$.\\
As we have found that some of our results are valid for a class
of finite $p$-groups larger  than $pE$-groups, we introduce the
following class of $p$-groups for every prime number $p$:
\begin{defn}
A finite $p$-group $G$ is called a $p\mathcal{E}$-group if $G$ is
a $2$-Engel group and there exists a non-negative integer $r$ such
that $\Omega_r(G)\leq Z(G)$ and {\rm
$\text{exp}(\frac{G}{G'})=p^r$}.
\end{defn}
\begin{rem}\label{Q_8} {\rm We know that a finite $pE$-group is a $p\mathcal{E}$-group
\cite{m3}; but the converse is false in general as one can see
that $Q_8$ and the group
\begin{align*}G=\langle a_1,a_2,a_3: [a_i,a_j,a_k]=1, (i,j,k\in\{1,2,3\}),
a_1^4=1,
 a_1^2=a_2^2=a_3^2=[a_1,a_2], [a_2,a_3]^2=[a_1,a_3]^2=1 \rangle,
 \end{align*}
are $2\mathcal{E}$-groups which are not  $2E$-groups. The
quaternion group $$Q_8=\langle a,b \;|\;
a^b=b^{-1},a^2=b^2,a^4=1\rangle$$ is not an $E$-group, since
$[a,b]\not=1$ and the map $\alpha$ which sends $a$ to $b$ and $b$
to $a$, can be extended to an automorphism  of $Q_8$.}
\end{rem}
\begin{lem}\label{2.4}
If $G$ is a finite $2$-Engel $p$-group, $p>2$ and
$m\in\mathbb{N}$, then $G^{p^{m}}=\{g^{p^{m}}:g\in G\}$
 and $G$ is regular.
\end{lem}
\begin{proof}
This follows easily from the following fact which is valid for all
$a,b\in G$:
$$a^{p^{m}}b^{p^{m}}=(ab)^{p^{m}}[a,b]^{\frac{p^{m}(p^m-1)}{2}}=\big(ab[a,b]^{\frac{p^m-1}{2}}\big)^{p^{m}}.$$
\end{proof}
\begin{lem}\label{2.5}
Let $G$ be a  $p\mathcal{E}$-group such that
{\rm $\text{exp}(\frac{G}{G'})=p^{r}$}. \\
{\rm (i)}\; {\rm $\text{exp}(G')=\text{exp}(G/Z(G))$} and {\rm $\text{exp}(G)=p^r\text{exp}(G')$}.\\
{\rm (ii)}\; if {\rm $\text{cl}(G)=2$}, then {\rm $\text{exp}(G')\leq p^r$}.\\
{\rm (iii)}\; if {\rm $\text{cl}(G)=3$}, then $p=3$ and {\rm
$\text{exp}(G')=3^{r+1}$}.
 \end{lem}
\begin{proof}
{\rm (i)}\; We have
\begin{align*}
  \text{exp}(G/Z(G)) \;\;\text{divides}\;\; n & \Leftrightarrow [a^n,b]=1\;\; \forall \; a,b\in G\\
   &  \Leftrightarrow [a,b]^n=1\;\; \forall \; a,b\in G\\
     &  \Leftrightarrow \text{exp}(G') \;\text{divides}\;\; n.
\end{align*}
This shows that $\text{exp}(G')=\text{exp}(G/Z(G))$.  For the
second part of (i), note that if $\text{exp}(G')=p^t$, then
$G^{p^{r+t}}\leq (G^{p^r})^{p^t}\leq(G')^{p^t}=1$. Also if
$G^{p^{r+t-1}}=1$, then $G^{p^{t-1}}\leq \Omega_r(G)\leq Z(G)$
and so $\text{exp}(G')\leq
p^{t-1}$, which is impossible. It follows that $\text{exp}(G)=p^{r+t}$.\\
(ii)\; This follows from (i) and the fact that $G'\leq Z(G)$.\\
(iii)\;  Note that in every 2-Engel group $K$, $\gamma_3(K)^3=1$
(see \cite[Theorem 12.3.6]{r1}). Thus, since $\text{cl}(G)=3$, we
have $p=3$ and one can write
$$(G')^{3^{r+1}}=[G^{3^r},G]^3\leq
[G',G]^3=\gamma_3(G)^3=1$$ which gives $\text{exp}(G')\leq
3^{r+1}$. If $\text{exp}(G')\leq 3^{r}$, then $G'\leq
\Omega_r(G)\leq Z(G)$ which is not possible, as $\text{cl}(G)=3$.
Therefore
 $\text{exp}(G')=3^{r+1}$.
\end{proof}
\begin{thm}\label{p-epsilon-2-gen}
Every  $2$-generator $p\mathcal{E}$-group is either abelian or
isomorphic to $Q_8$.
\end{thm}
\begin{proof}
Suppose that $G =\langle a, b\rangle$ is a $p\mathcal{E}$-group
and that $\text{exp}(G/G')=n=p^r$. Then $a^n, b^n \in G'=\langle
[a, b]\rangle$ and since $G'$ is a cyclic $p$-group, we have
$\langle a^n,b^n\rangle=\langle a^n\rangle$ or $\langle
b^n\rangle$. Without lose of generality we can suppose that
$b^n=a^{ns}$ for some integer $s$. Then
$$(ba^{-s})^n= [b, a]^{sn(n-1)/2}\eqno{(1)}$$ which by Lemma \ref{2.5} is
trivial if $p$ is odd or if $p=2$ and either $\text{exp}(G')\leq
2^{r-1}$ or $2|s$. In any of these cases we would have that
$ba^{-s}$ is in $Z(G)$  and $G$ is abelian. So one can suppose
that $p=2$, that $s$ is odd and that $\text{exp}(G')=2^r$. In
this case the equality $(1)$ implies that $(ba^{-s})^{2n}=1$ and
since $G$ is a $2\mathcal{E}$-group, we have $(ba^{-s})^2\in
Z(G)$. Thus
$$1=[(ba^{-s})^2,a]=[ba^{-s},a]^2=[b,a]^2.$$ It follows that
$r\leq 1$ and so   $|G|=|\frac{G}{G'}||G'|\leq 8$  and $G$ is
either abelian or $G\cong Q_8$ or the dihedral group $D_8$ of
order 8. But $D_8$ is not a $2\mathcal{E}$-group, since there are
elements of order 2 in $D_8$ which are not central. On the other
hand $Q_8$ is a $2\mathcal{E}$-group, since
$\text{exp}(Q_8/Q_8')=2$ and the only element of order 2 in $Q_8$
is central.   This completes the proof.
\end{proof}
\begin{lem}\label{Z-2} Let $G$ be a $p\mathcal{E}$-group and let
{\rm $\text{exp}(\frac{G}{G'})=p^{r}$}. Then
$Z_2(G)^{p^r}=Z(G)\cap G^{p^r}$. In particular if {\rm
$\text{cl}(G)=3$}, then {\rm
$\text{exp}(\frac{Z_2(G)}{Z(G)})=3^r$}.
\end{lem}
\begin{proof}
By \cite[Theorem 12.3.6]{r1}, we may   assume that $p=3$. Let
$x\in Z_2(G)^{3^r}$. By Lemma \ref{2.4}, $x=y^{3^r}$ for some
$y\in Z_2(G)$. Since  $G^{3^r}\leq G'\leq Z_2(G)$, we have
$$[x,g]=[y^{3^r},g]=[y,g^{3^r}]=1$$ for all $g\in G$. This implies that  $x\in Z(G)$.\\
Now assume that $x\in G^{3^r}\cap Z(G)$. Then  $x=y^{3^r}$ for
some $y\in G$ and so $$1=[x,g]=[y^{3^r},g]=[y,g]^{3^r}$$ for all
$g\in G$. Then $[y,g]\in \Omega_r(G)\leq Z(G)$ which implies that
$y\in Z_2(G)$. Hence $x\in Z_2(G)^{3^r}$. This completes the proof of the first part.\\
If $\text{cl}(G)=3$ and $Z_2(G)^{3^{r-1}}\leq Z(G)$, then we have
$$G^{3^r}=(G^3)^{3^{r-1}}\leq Z_2(G)^{3^{r-1}}\leq Z(G),$$
(note that since $\gamma_3(G)^3=1$, $G^3\leq Z_2(G)$). Thus
$G^{3^r}\leq Z(G)$ which is impossible by  Lemma \ref{2.5}(ii).
 Therefore $\text{exp}(\frac{Z_2(G)}{Z(G)})=3^r$.
\end{proof}
\begin{rem}\label{2.8}
{\rm If $G$ is a  $2$-Engel group, then $G^3G'\leq Z_2(G)$. This
is because, $\text{cl}(G)\leq 3$ and $\gamma_3(G)^3=1$ (see
\cite[Theorem 12.3.6]{r1}). Thus if $G$ is a finite $2$-Engel
$3$-group, then we always have that $\Phi(G)\leq Z_2(G)$.}
\end{rem}
\begin{lem}\label{2.9}
 Let $G$ be a $2$-Engel group. If $\frac{G}{Z_2(G)}$ is $2$-generator, then
{\rm $\text{cl}(G)\leq 2$}.
\end{lem}
\begin{proof}
If $\frac{G}{Z_2(G)}= \langle aZ_2(G), bZ_2(G)\rangle$, then $G
=\langle a,b, Z_2(G)\rangle$. Thus
$$G'=\langle[x,y] ,\gamma_3(G) \;|\; x,y \in \{a,b\}\cup Z_2(G)\rangle.$$
Since $G$ is 2-Engel, we have $G'\leq Z(G)$. This completes the
proof.
\end{proof}
\begin{thm}\label{3e}
Every  $3$-generator $p\mathcal{E}$-group is nilpotent of class
at most $2$.
\end{thm}
\begin{proof}
For a contradiction suppose that $G =\langle x, y, z \rangle$ is a
$p\mathcal{E}$-group of class 3. Suppose that
$\text{exp}(G/G')=3^r$. Let $H=(G')^3 \gamma_3(G)$. Notice that,
by Lemma \ref{2.5}, $[H,G]=H^{3^r}=1$. Modulo $H$ we have that
$$
  x^{3^r}=[x, y]^{\alpha}[y,z]^{t_1}[z, x]^{\beta}, \;\;
  y^{3^r}=[x, y]^{\gamma} [y, z]^{\beta'} [z,x]^{t_2}, \;\;
   z^{3^r}=[x,y]^{t_3}[y, z]^{\alpha'}[z, x]^{\gamma'},$$
 for some integers $\alpha, \beta,\gamma,\alpha',\beta',\gamma',t_1,t_2,t_3 \in \{-1,
0, 1\}$. Since $[x,y,z]\not=1$, it follows that $t_i$'s must all
be zero. Now since $[x^{3^r},y]=[x,y^{3^r}]$, one can see that
$\beta'=-\beta$. Similarly one can deduce that $\gamma'=-\gamma$
and $\alpha'=-\alpha$. Therefore, modulo $H$ we  have
$$  x^{3^r}=[x, y]^{\alpha}[z, x]^{\beta}, \;\;
  y^{3^r}=[x, y]^{\gamma} [y, z]^{-\beta} \;\;
   z^{3^r}=[y, z]^{-\alpha}[z, x]^{-\gamma}.
$$
  It follows that
$$
[x, y]^{3^r}=[x, y, z]^{\beta}, \;\;
 [z, x]^{3^r}=[x, y, z]^{-\alpha}, \;\;
 [y, z]^{3^r} = [x, y, z]^\gamma.$$
 Then $x^{3^{2r}} = y^{3^{2r}}=z^{3^{2r}}=1$. Now since $G$ is regular, it follows that  $G^{3^{2r}}= 1$
that contradicts Lemma \ref{2.5}.
\end{proof}
 \begin{thm}\label{3e-gen}
Let $G$ be a non-abelian  $3$-generator $p\mathcal{E}$-group,
{\rm $\text{exp}(\frac{G}{G'})=p^r$}, {\rm $\text{exp}(G')=p^t$}
and $p>2$. Then $|G|=p^{3(r+t)}$ and $G$ has the following
presentation
\begin{align*} \langle x,y,z \;|\;
x^{p^{r+t}}=y^{p^{r+t}}=z^{p^{r+t}}=[x^{p^t},y]=[x^{p^t},z]=
[y^{p^t},x]=[y^{p^t},z]=[z^{p^t},x]=[z^{p^t},y]=1, &\\
 [x,y]=x^{p^rt_{11}}y^{p^rt_{12}}z^{p^rt_{13}},
[x,z]=x^{p^rt_{21}}y^{p^rt_{22}}z^{p^rt_{23}},
[y,z]=x^{p^rt_{31}}y^{p^rt_{32}}z^{p^rt_{33}} \rangle,
\end{align*} where $1\leq t\leq r$ and $[t_{ij}]\in
GL(3,\mathbb{Z}_{p^t})$. Moreover  every group with the above
presentation is a $p\mathcal{E}$-group.
 \end{thm}
\begin{proof}
By Theorem \ref{3e}, $\text{cl}(G)=2$. Since $G$ is 3-generator,
there exist elements $a,b,c \in G$ such that
$\frac{G}{Z(G)}=\langle aZ(G) \rangle \times \langle bZ(G)
\rangle \times \langle cZ(G) \rangle$. Thus, since $G'$ has
exponent $p^t$,  we have $|aZ(G)|=|bZ(G)|=p^t$ and  $|cZ(G)|=p^s$
for some integers $t,s\geq 0$ with $t\geq s$. We also have
$G'=\langle [a,b],[a,c],[b,c]\rangle$, since $\text{cl}(G)=2$ and
$G=\langle a,b,c,Z(G) \rangle$. Since $|aZ(G)|=p^t$ and
$|cZ(G)|=p^s$, we have $|[a,b]|\leq p^t$, $|[a,c]|\leq p^s$ and
$|[b,c]|\leq p^s$. Therefore $|G'|\leq p^{t+2s}$. Also, since $G$
is regular, $|G:\Omega_r(G)|=|G^{p^r}|$. Then $|G|\leq
|\Omega_r(G)||G^{p^r}|\leq |Z(G)||G'|$ and so $|G:Z(G)|\leq |G'|$.
Hence $p^{2t+s}\leq p^{t+2s}$ and $t\leq s$. It follows that
$s=t$, $|G'|=|\frac{G}{Z(G)}|=p^{3t}$ and $G'=\langle [a,b]\rangle
\times \langle [a,c]\rangle \times \langle [b,c]\rangle$. Since
$G$ is not abelian, $t\geq 1$. Thus $ \frac{G}{G^pZ(G)}\cong C_p
\times C_p \times C_p$. This implies
that $G=\langle a,b,c\rangle$.\\
Now, since $G^{p^r}\leq G'$ and $|G'|=|G:Z(G)|\leq
|G:\Omega_r(G)|=|G^{p^r}|$, we have $G'=G^{p^r}$. Also we have
$G^{p^r}=\langle a^{p^r},b^{p^r},c^{p^r}\rangle$ (since $t\leq
r$). By Lemma  \ref{2.5} $\text{exp}(G)=p^{r+t}$ and since
$G'=G^{p^r}$ is an abelian group of  order $p^{3t}$ it follows
that $G^{p^r}=\langle a^{p^r}\rangle \times \langle b^{p^r}\rangle
\times \langle c^{p^r}\rangle$, $t\leq r$ and
$|a|=|b|=|c|=p^{r+t}$. Also since $G^{p^t}=\langle
a^{p^t},b^{p^t},c^{p^t}\rangle$ and $G^{p^r}=\langle
a^{p^r}\rangle \times \langle b^{p^r}\rangle \times \langle
c^{p^r}\rangle\leq G^{p^t}$, it is not hard to see that
$G^{p^t}=\langle a^{p^t}\rangle \times \langle b^{p^t}\rangle
\times \langle c^{p^t}\rangle$ and so
$$p^{3r}=|G^{p^t}|\leq |\Omega_r(G)|\leq |Z(G)|=|G:G'|\leq p^{3r}.$$
It follows that $G^{p^t}=\Omega_r(G)=Z(G)$,
$|\Omega_t(G)|=|G:G^{p^t}|=|G'|$ and so $G'=\Omega_t(G)$. Thus we
have the following information about $G$:
\begin{align*}
&|G|=p^{3(r+t)}, \text{exp}(G)=p^{r+t}, G=\langle a,b,c\rangle, \\
& |a|=|b|=|c|=p^{r+t} \\ &Z(G)=\Omega_r(G)=G^{p^t}=\langle
a^{p^t}\rangle \times \langle b^{p^t}\rangle \times \langle
c^{p^t}\rangle \\ & G'=\Omega_t(G)=G^{p^r}=\langle a^{p^r}\rangle
\times \langle b^{p^r}\rangle \times \langle c^{p^r}\rangle
 \end{align*}
Hence  there exists a $3\times 3$ matrix $T=[t_{ij}]\in
GL(3,\mathbb{Z}_{p^t})$ such that
\begin{align*}
 [a,b]=a^{p^rt_{11}}b^{p^rt_{12}}c^{p^rt_{13}}&\\
[a,c]=a^{p^rt_{21}}b^{p^rt_{22}}c^{p^rt_{23}}&\\
[b,c]=a^{p^rt_{31}}b^{p^rt_{32}}c^{p^rt_{33}}
\end{align*}
and every element of $G$ can be written as $a^ib^jc^k$ for some
$i,j,k\in \mathbb{Z}$ and
\begin{align*}
(a^ib^jc^k)(a^{i'}b^{j'}c^{k'})=a^{i+i'-i'jp^rt_{11}-i'kp^rt_{21}-j'kp^rt_{31}}&  \\
  b^{j+j'-i'jp^rt_{12}-i'kp^rt_{22}-j'kp^rt_{32}}
c^{k+k'-i'jp^rt_{13}-i'kp^rt_{23}-j'kp^rt_{33}}
\end{align*}
Now consider $\widetilde{G}=\mathbb{Z}_{p^{r+t}}\times
\mathbb{Z}_{p^{r+t}}\times \mathbb{Z}_{p^{r+t}}$ and define the
following binary operation on $\widetilde{G}$:
\begin{align*}
(i,j,k)(i',j',k')= \big(
i+i'-i'jp^rt_{11}-i'kp^rt_{21}-j'kp^rt_{31},
  j+j'-i'jp^rt_{12}-i'kp^rt_{22}-j'kp^rt_{32},&\\
k+k'-i'jp^rt_{13}-i'kp^rt_{23}-j'kp^rt_{33} \big)
\end{align*}
It is easy to see that $\widetilde{G}$ with this binary operation
is a group and $G\cong \widetilde{G}$. Now one can easily see
that the group $G$ has the required presentation.
\end{proof}
\begin{thm}\label{autabelian} {\rm(}The main result of \cite{mm2}{\rm)} For $p$ an odd prime, there
exists no finite non-abelian $3$-generator $p$-group having an
abelian automorphism group.
\end{thm}
\begin{thm}\label{3-gen}
Let $G$ be a non-abelian finite $3$-generator $pE$-group and
$p>2$. Then {\rm $\text{exp}(G') < \text{exp}(\frac{G}{G'})$}.
\end{thm}
\begin{proof}
Suppose, for a contradiction, that $\text{exp}(G') \geq
\text{exp}(\frac{G}{G'})$. By Lemma \ref{2.5},
$\text{exp}(G')=\text{exp}(\frac{G}{G'})=p^r$ and by the proof of
Theorem \ref{3e-gen}, we have $G=\langle a,b,c\rangle$ and
$$Z(G)=\Omega_r(G)=G^{p^r}=G'=\langle [a,b]\rangle \times \langle
[a,c]\rangle \times \langle [b,c]\rangle, |a|=|b|=|c|=p^{2r},
|[a,b]|=|[a,c]|=|[b,c]|=p^r.$$
If we prove that $Aut(G)$ is abelian, then  Theorem \ref{autabelian} completes the proof.\\
Let $\alpha\in Aut(G)$. There exist integers $i,n,m$ and an
element $w\in G'$ such that $a^\alpha=a^ib^nc^mw$. Since
$[a^\alpha,a]=1$, we have $[b,a]^n[c,a]^m=1$ and so $n\equiv
m\equiv 0$  (mod $p^r$). Therefore $a^\alpha=a^iw_a$ and similarly
$b^\alpha=b^jw_b$, $c^\alpha=c^kw_c$, where $1\leq i,j,k\leq
p^r-1$ and $w_a,w_b,w_c\in G'=Z(G)$. From $[(ab)^\alpha,ab]=1$
and $[(ac)^\alpha,ac]=1$, it follows respectively   that  $i=j$
and $i=k$. Also from equality $G^{p^r}=G'$, we have
$a^{p^r}=[a,b]^t[b,c]^r[a,c]^s$ where $t,r$ and $s$ are integers.
Then
$(a^\alpha)^{p^r}=[a^\alpha,b^\alpha]^t[b^\alpha,c^\alpha]^r[a^\alpha,c^\alpha]^s$
and we obtain $a^{p^ri}=x^{p^ri^2}$. Therefore $i^2\equiv i $ (mod
$p^r$) and so $i=1$. Therefore all automorphisms of $G$ are
central so that they fix the elements of $G'=Z(G)$. If $\alpha,
\beta\in Aut(G)$, then $x^{\alpha\beta}=x^{\beta\alpha}$ for every
$x\in \{a,b,c\}$. Hence $Aut(G)$ is abelian which contradicts
Theorem \ref{autabelian}.
\end{proof}
\begin{cor}\label{exp}
Let $G$ be a finite $3$-generator $pE$-group and $p>2$. If {\rm
$\text{exp}(G) \leq p^2$}, then $G$ is abelian.
\end{cor}
\begin{proof}
If $G$ is non-abelian, then by Lemma \ref{2.5} we have
$\text{exp}(G')=\text{exp}(\frac{G}{G'})=p$ which contradicts
Theorem \ref{3-gen}.
\end{proof}
\begin{lem}\label{lep^5}
 Let $G$ be a  $p\mathcal{E}$-group such that  $|G|\leq
p^5$. Then $G$ is abelian or $G$ is isomorphic to one of the
following groups:
$Q_8$, $Q_8\times C_2$, $Q_8\times C_2\times C_2$,\\
$\left< x, y, z \;|\; x^4=y^4=[y,z]=1, x^2=z^2=[x,y],
(xz)^2=y^2\right>$ or\\
$\left< x, y, z \;|\; x^4=z^4=[y,z]=1, x^2=y^2=[x,y],
[x,z]=z^2\right>$
\end{lem}
\begin{proof} It can be easily checked by {\sf GAP} \cite{g} that  there are
exactly five non-abelian  $2\mathcal{E}$-groups, which are the
same as listed in the lemma. Thus we may assume that $p>2$ and $G$
is non-abelian. If $d(G)\geq 4$, then $|Z(G)|\geq
|\Omega_1(G)|=|G:G^p|\geq p^4$  (since $G$ is regular) and so
$|G:Z(G)|\leq p$, which is impossible. Hence $d(G)=3$ and so
$|G|\geq p^4$ which contradicts Theorem \ref{3e-gen}.
\end{proof}
\begin{rem}\label{PE}
Suppose that  $G$ is a finite $p$-group such that $\Omega_1(G)\leq
Z(G)$. If $G$ has no  non-trivial abelian direct factor, then
$\Omega_1(G)\leq \Phi(G)$. To see this, let $x\in G$ be of order
$p$ and $x\not\in \Phi(G)$. Then there exists a maximal subgroup
$M$ such that $x\not\in M$. Since $x\in Z(G)$, we have
$\left<x\right>\trianglelefteq G$, so that
$G=M\times\left<x\right>$,  a contradiction.
\end{rem}
\begin{thm}\label{p^7-ep}
Let $G$ be a  $p\mathcal{E}$-group having no abelian direct
factor and $p>2$. If $|G|=p^7$, then $G$ is abelian.
\end{thm}
\begin{proof}
Suppose, for a contradiction, that  $G$ is not abelian. If
$d(G)\geq 4$, by Remark \ref{PE}, we have
 $$|\Phi(G)|\geq|\Omega_1(G)|=|G:G^p|\geq |G:\Phi(G)|\geq p^4$$
  and so $|G|\geq p^8$ which is impossible. Therefore, by Lemma
\ref{p-epsilon-2-gen}, we have $d(G)=3$ which is a contradiction
by Theorem \ref{3e-gen}.
\end{proof}
\begin{lem}\label{E-inf}
Let $G$ be an $E$-group and $a\in G$ be such that
$\left<aG'\right>$ is an infinite direct summand of
$\frac{G}{G'}$. Then $a\in Z(G)$.
\end{lem}
\begin{proof}
By assumption, we have  $\frac{G}{G'}=\langle aG'\rangle \oplus
\langle XG'\rangle$ for some $X\subseteq G$. Since $G$ is
nilpotent, $G'\leq \Phi(G)$ and so $G=\langle a,X\rangle$.
Therefore it is enough to show that $[a,x]=1$ for all $x\in X$.\\
Let $\pi:G\rightarrow \frac{G}{G'}$ be the natural epimorphism
and $\psi:\langle aG' \rangle \oplus\langle XG'\rangle
\rightarrow \langle aG'\rangle$  the projection map on the first
component. Now  for each $x\in X$ let $\varphi_x: \langle
aG'\rangle \rightarrow \langle x\rangle$ be the map defined by
$a^iG'\mapsto x^i$ for all $i\in\mathbb{Z}$. Since
$\left<aG'\right>\cong\mathbb{Z}$, $\varphi_x$ is a group
homomorphism mapping $aG'$ to $x$. Thus $\pi\psi\varphi_x$ is an
endomorphism of $G$ mapping $a$ to $x$. Since $G$ is an $E$-group,
we have that $[a,x]=1$. This completes the proof.
\end{proof}
\begin{thm}
Let $G$ be an infinite finitely generated  $E$-group. Then
$G=K\times H$, where $K$ is a central torsion-free subgroup of $G$
and $H$ is a finite subgroup of $G$. In particular if $G$ is
infinite and indecomposable then $G$ is infinite cyclic.
\end{thm}
\begin{proof}
Since $G$ is infinite and nilpotent, $\frac{G}{G'}$ is an
infinite finitely generated group. Thus $$\frac{G}{G'}=\langle
a_1G' \rangle\oplus \cdots \oplus\langle a_nG' \rangle \oplus
\langle b_1G' \rangle \oplus\cdots  \oplus \langle b_mG' \rangle
\eqno{(*)}
$$ for some $a_1,\dots,a_n,b_1,\dots, b_m \in G$ such that
$\langle a_iG'\rangle$ is infinite, $\langle b_iG'\rangle$ is
finite  and $G=\langle a_1,\dots,a_n,b_1,\dots,b_m \rangle$. By
Lemma \ref{E-inf}, $K=\langle a_1,\dots,a_n\rangle \leq Z(G)$. It
follows that $G'=H'$, where $H=\langle b_1,\dots,b_m\rangle$. Thus
$\frac{H}{H'}$ is finite and since $H$ is a nilpotent group, $H$
is  finite.  Since $K\leq Z(G)$ and we have the decomposition
$(*)$, $K$ is a torsion-free group. It follows that $G=K \times
H$. Now if $G$ is indecomposable, we must have $H=1$ and $n=1$.
This completes the proof.
\end{proof}
\begin{lem} \label{z2}
Let $G$ be a finite nilpotent $p$-group of class $3$. If
$|G':G'\cap Z(G)|=p$, then $|G:Z_2(G)|=p^2$.
\end{lem}
\begin{proof}
Suppose $G = \langle a, b, c_1, \dots , c_r\rangle$ where $G'Z(G)
= \langle [a, b] \rangle Z(G)$. By replacing $c_i$ by a suitable
$c_ia^{\alpha_i}b^{\beta_i}$ one can assume that $[c_i, a], [c_i,
b] \in Z(G)$ for $i=1,\dots, r$. We claim that $c_1,\dots, c_r \in
Z_2(G)$. For this it suffices to show that $[c_i, c_j ] \in Z(G)$
for $1\leq i < j \leq r$. Suppose $$[c_i, c_j ]=[a, b]^rz$$ with
$z\in Z(G)$. As $$1=[a, b, c_k][b, c_k, a][c_k, a, b] = [a, b,
c_k],$$ this is clear. Hence $G/Z_2(G) = \langle a, b\rangle
Z_2(G)/Z_2(G)$ is of order $p^2$.
\end{proof}
\section{\bf Proofs of the main results}

\noindent{\bf Proof of Theorem \ref{thm3gen}.} Let $G$ be a
3-generator $E$-group. If $\frac{G}{G'}$ is finite, since $G$ is
nilpotent, $G$ is finite and $G$ is a direct product of its Sylow
subgroups. Every Sylow subgroup of $G$ is endomorphic image of
$G$ and so by \cite{m3} they are at most 3-generator $E$-groups.
In this case, Theorem \ref{3e} completes the
proof.\\
If $\frac{G}{G'}$ is infinite, then by the fundamental theorem of
finitely generated abelian groups, we have $\frac{G}{G'}=\langle
aG'\rangle \oplus \langle bG',cG'\rangle$ for some $a,b,c\in G$
such that $aG'$ has infinite order. Thus by Lemma \ref{E-inf},
$a\in Z(G)$ and  since $G$ is 2-Engel, it follows easily that
$G'=\langle [b,c] \rangle$ and since $G$ is 2-Engel,
$\gamma_3(G)=\gamma_3(\langle b,c \rangle)=1$.  This completes the proof. \hfill$\Box$\\

\noindent {\bf Proof of Theorem \ref{thm3^{10}}.} Suppose, for a
contradiction, that  $G$ is a finite $3E$-group of the least order
subject to the properties $\text{cl}(G)=3$ and $|G|\leq 3^{10}$.
Then $G$ is indecomposable and so $\Omega_1(G)\leq \Phi(G)$, by
Remark \ref{PE}. Thus $\Omega_1(G)\leq \Phi(G)\cap Z(G)$. If
$d(G)\geq 5$, then $|\frac{G}{\Phi(G)}|\geq 3^5$. Since $G$ is
regular and $\Phi(G)=G^3G'$,  $|\Phi(G)\cap Z(G)|\geq
|\Omega_1(G)|=|G:G^3|\geq 3^5$ and  since $\text{cl}(G)=3$,
$\Phi(G) \cap Z(G) \lvertneqq \Phi(G)$. It follows that  $|G|\geq
3^{11}$, a contradiction. Thus  Theorem \ref{thm3gen} implies
that $d(G)=4$. If $|G':Z(G)\cap G'|=3$, then by Lemma \ref{z2},
we have $|G:Z_2(G)|=9$. Therefore $\frac{G}{Z_2(G)}$ is a
2-generator group and by Lemma \ref{2.9}, $\text{cl}(G)\leq 2 $, a
contradiction. Hence $|G':Z(G)\cap G'|\geq 9$. Since
$$|G|=|G:\Phi(G)||\Phi(G):\Phi(G)\cap Z(G)||\Phi(G)\cap Z(G)|$$ and
$$|\Phi(G):\Phi(G)\cap Z(G)|=|G'G^3:G'G^3 \cap
Z(G)|=|Z(G)G'G^3:Z(G)|=\frac{|Z(G)G'||G^3|}{|Z(G)G' \cap
G^3||Z(G)|},$$
 we have
 $$|G|=|G:\Phi(G)||G':G'\cap Z(G)||G^3:G'Z(G)\cap G^3||\Phi(G)\cap Z(G)|\geq
 3^{10}.$$
 Thus $|G|=3^{10}$ , $|\Omega_1(G)|=|\Phi(G)\cap Z(G)|=3^4$, $|G':Z(G)\cap G'|=
 9$ and $G^3\leq G'Z(G)$.
 Since $|G:G^3|=|\Omega_1(G)|$, $G^3=\Phi(G)\leq
 Z_2(G)$, $\Phi(G)$ is an abelian group of order $3^6$ and $d(\Phi(G))=4$. Hence
$$\Phi(G)\cong C_{27}\times C_3\times C_3\times C_3 \;\; \text{or} \;\; C_9\times
C_9\times C_3\times C_3.$$ Also we have $G^{'9}=[G^3,G]^3\leq
(\gamma_3(G))^3=1$ and  Lemma \ref{2.5}(ii) yields that
 $\text{exp}(\frac{G}{G'})=3$. Hence by Lemma \ref{Z-2}, we have $Z_2(G)^3 = \Phi(G)\cap Z(G)$.
Now Lemma \ref{2.9} implies that $d(\frac{G}{Z_2(G)})=3$ or $4$.
Then $|Z_2(G)|=3^6$ or $3^7$, which implies that $Z_2(G)=\Phi(G)$
or $|Z_2(G):\Phi(G)|=3$. If $Z_2(G)=\Phi(G)$, then $|Z_2(G)^3| =
|\Phi(G)\cap Z(G)|=9$,  a contradiction. Thus
$|Z_2(G):\Phi(G)|=3$. Since $[Z_2(G),\Phi(G)]=1$, we have
$\Phi(G)\leq Z(Z_2(G))$, $Z_2(G)$ is an abelian group and
$d(Z_2(G))=4$. Hence $$Z_2(G)\cong C_{81}\times C_3\times
C_3\times C_3 \;\; \text{or} \;\; C_{27}\times C_9\times C_3\times
C_3 \;\;\text{or}\;\; C_9\times C_9\times C_9\times C_3.$$ Thus
$|Z_2(G)^3| = |\Phi(G)\cap Z(G)|=27$. This contradiction completes
the proof. \hfill $\Box$\\

\noindent{\bf Proof of Theorem \ref{2-gen}.} $i)$ \; Let $G$ be a
 2-generator $E$-group. Suppose first that $\frac{G}{G'}$ is
finite. Since $G$ is nilpotent, $G$ is finite and so it is the
direct product of its Sylow subgroups. Every Sylow subgroup of $G$
is also at most 2-generator and an $E$-group.
Now Theorem \ref{p-epsilon-2-gen} and Remark \ref{Q_8} imply that every Sylow subgroup of $G$ is abelian and so $G$ is abelian.\\
Therefore we may assume that $\frac{G}{G'}$ is infinite. It
follows from the fundamental theorem of finitely generated
abelian groups, that $\frac{G}{G'}=\langle aG' \rangle
\oplus\langle bG' \rangle$ for some $a,b\in G$ such that $\langle
aG'\rangle$ is infinite.  Since $G'\leq \Phi(G)$, $G=\langle
a,b\rangle$ and
Lemma \ref{E-inf} completes the proof of (i).\\
$ii)$ \; Let $G$ be an infinite  3-generator $E$-group. Since $G$
is infinite and nilpotent, $\frac{G}{G'}$ is an infinite finitely
generated 3-generator group. Thus $\frac{G}{G'}=\langle aG'
\rangle \oplus\langle bG' \rangle \oplus \langle cG' \rangle$ for
some $a,b,c\in G$ such that $\langle aG'\rangle$ is infinite and
$G=\langle a,b,c \rangle$. By Lemma \ref{E-inf}, $a\in Z(G)$. If
either $\langle bG' \rangle$ or  $\langle cG' \rangle$ is
infinite, then Lemma \ref{E-inf} implies that $G$ is abelian. Thus
we may assume that  $\langle bG' \rangle$ and   $\langle cG'
\rangle$ are both finite. It follows that $G'=\langle [b,c]\rangle
\leq H=\langle b,c \rangle$ is finite. Thus $G=\langle a \rangle
\times H$, since $a\in Z(H)$ is of infinite order and $H$ is a
finite group. Hence $H$ is a 2-generator $E$-group, as it is a
direct factor of $G$. Now part (i)
completes the proof.   ~\hfill $\Box$\\

\noindent{\bf Proof of Theorem \ref{thmp^6}.} Suppose, for a
contradiction, that $G$ is a non-abelian $pE$-group of order
$p^6$.  We  see that non-abelian groups in Lemma \ref{lep^5} are
not $E$-groups and so $|G|=p^6$. If $p=2$, then  one can see
(e.g., by {\sf GAP} \cite{g}) that there exist ten
$p\mathcal{E}$-groups $T$. We have checked by the package {\tt
AutPGrp} in {\sf GAP} \cite{g}, that for each such a group $T$,
there are  $\alpha\in Aut(T)$ and $x\in T$ such that
$[x,x^{\alpha}]\not=1$ (the automorphism $\alpha$ is in a set of
generators given by the package for $Aut(T)$); thus they are not
$E$-groups. Therefore $p$ is odd. Similar by the proof of Theorem
\ref{p^7-ep}, we have $d(G)=3$ (since $G$ has no non-trivial
abelian direct factor). Then by Theorem \ref{3e-gen}, we have $
\text{exp}(G)=p^2$. Hence by Corollary \ref{exp}, the proof is
complete.
\hfill$\Box$\\

\noindent{\bf Proof of Theorem \ref{thmp^7}.} $i)$ Suppose, for a
contradiction, that $G$ is a non-abelian $pE$-group of least order
subject to the property  $|G|\leq p^7$.   Then by Theorem
\ref{thmp^6}, $|G|=p^7$. By the choice of $G$ and that every
direct factor of an $E$-group is again an $E$-group, $G$ has no
abelian direct factor. Now Theorem \ref{p^7-ep} completes the proof.\\
\noindent
 $ii)$ Let
 \begin{align*} G=\langle x,y,z,t \;|\;
 y^2=z^2,[x,y]=[x,t]=y^2,[x,z]=z^2t^2,  [y,z]=x^2,[y,t]=[z,t]=[y,z,t]=1\rangle.
 \end{align*}
We have $G'=Z(G)=G^2=\Omega_1(G)=\{g^2 \;|\; g\in G\},
\text{exp}(G)=4$ and $|G|=2^7$. By Package {\tt AutPGrp} of {\sf
GAP} \cite{g}, $Aut(G)$ is abelian and so every automorphism of
$G$ is central. Then for all $\beta\in
Aut(G)$ and all $a\in G$, $[a,a^\beta]=1$.\\
Now we prove that every  endomorphism $\alpha$ of $G$ which is not
an automorphism,  maps $G$ into $Z(G)$. We denote by  $Ker\alpha$
and $Im\alpha$ the kernel and image of $\alpha$, respectively. We
first show that if $y^2\in Ker\alpha$, then $Im\alpha\leq Z(G)$.
Since $(y^\alpha)^2=(z^\alpha)^2=1$, we have $y^\alpha$ and
$z^\alpha\in Z(G)$. Then $(x^2)^\alpha=[y^\alpha,z^\alpha]=1$ and
so $x^\alpha\in Z(G)$. Also
$(t^2)^\alpha=(z^2t^2)^\alpha=[x^\alpha,z^\alpha]=1$ which
implies that  $t^\alpha\in Z(G)$. Therefore in this case,
$Im\alpha\leq Z(G)$.
Thus we can assume $y^2\notin Ker\alpha$.\\
Since $\alpha\notin Aut(G)$,  $Ker\alpha \cap Z(G) \not=1$. Now it
follows from the equality $G'=\Omega_1(G)=Z(G)=\{g^2 \;|\; g\in
G\}$ that there exists an element $g\in G$ such that $1\neq g^2\in
Ker\alpha$. Then $g^\alpha\in G'$ and $g\notin G'$. Thus
$g=x^iy^jz^kt^lw$, where $0\leq i,j,k,l\leq 1, w\in G'$ and at
least one of the integers $i,j,k,l$ is nonzero. For all $a\in G$,
$[a,g]^\alpha=[a^\alpha,g^\alpha]=1$ and so $[a,g]\in Ker\alpha$.
Since $[t,g]\in Ker\alpha$, $y^{2i}\in Ker\alpha$ and so $i=0$.
Also $[y,g]\in Ker\alpha$, implies that $x^{2k}\in Ker\alpha$. If
$k=1$ then $(x^2)^\alpha=1$ and so $x^\alpha\in Z(G)$. Therefore
$(y^2)^\alpha=[x^\alpha,t^\alpha]=1$ which is impossible.
Therefore $k=0$. Since $[z,g]\in Ker\alpha$ and $[x,g]\in
Ker\alpha$, we have $j=l=0$ which is a contradiction
and the proof is complete.\\
By a similar proof one can see that the following groups are
non-abelian $2E$-groups of order $2^7$.
\begin{align*}
\langle x,y,z,t \;|\;
 y^2=z^2t^2,[x,y]=[x,t]=z^2, [x,z]=t^2,[y,z]=x^2,
 [y,t]=[z,t]=[y,z,t]=1 \rangle.
\end{align*}
\begin{align*}
\langle x,y,z,t \;|\; y^2=z^2,[x,z]=y^2,[x,t]=t^2y^2,[x,y]=t^2,
[y,z]=x^2t^2,[y,t]=t^2,  [z,t]=[y,z,t]=1 \rangle.
\end{align*}
These groups are not isomorphic and their automorphism groups are
abelian and every endomorphism which is not an automorphism maps
the group into  its center. \hfill$\Box$\\

\noindent{\bf Acknowledgment.} The authors are grateful  to the
referee for his/her valuable comments which caused the paper to
be shorter.

\begin{thebibliography}{99}

\bibitem{c} A. Caranti, Finite $p$-groups of exponent $p^2$ in which each
element commutes with its endomorphic images, J. Algebra {\bf 97}
(1985) 1-13.

\bibitem{cf} A. Caranti, S. Franciosi, F. de Giovanni, Some examples
of infinite groups in which each element commutes with its
endomorphic images, Group theoy, Proc. Conf., Brixen/Italy
(1986), Lect. Notes Math. {\bf 1281} (1987) 9-17.

\bibitem{f} R. Faudree, Groups in which each element commutes with its
endomorphic images, Proc. Amer. Math. Soc. {\bf 27} (1971)
236-240.

\bibitem{g} The GAP Group, GAP-Groups, Algorithms, and Programming,
Version 4.4; 2005, (http://www.gap-system.org).


\bibitem{k} The Kourovka Notebook, Unsolved Problems in Group
Theory,   15th augm. ed., Novosibirisk, 2002.

\bibitem{l} F. W. Levi, Groups in which the commutator operation
 satisfies certain algebraic condition, J. Indian Math. Soc.
 {\bf 6}  (1942) 87-97.


\bibitem{m1} J. J. Malone, A non-abelian 2-group whose endomorphisms
generate a ring, and other examples of $E$-groups, Proc. Edinburgh
Math. Soc. {\bf 23} (1980) 57-59.

\bibitem{m2} J. J. Malone, Endomorphism near-rings through the ages, Kluwer
Academic Publishers. Math. Appl.,  Dordr. {\bf 336} (1995) 31-43.

\bibitem{m3} J. J. Malone, More on groups in which each element
commutes with its endomorphic images, Proc. Amer. Math. Soc. {\bf
65} (1977) 209-214.


\bibitem{mm2} M. Morigi, On the minimal number of generators of finite
non-abelian $p$-groups having an abelian automorphism group, Comm.
Algebra {\bf 23}, No. 6, (1995) 2045-2065.

\bibitem{r1} D. J. S. Robinson, A course in the theory of groups,
Second Edition, Springer-Verlag, New York, 1995.

\end{thebibliography}
\end{document}